\documentclass[11pt]{amsart}

\usepackage{enumitem}
\usepackage{lineno} 
\usepackage{multicol}
\usepackage{orcidlink}
\usepackage{amsthm,amssymb,mathrsfs,pstricks,booktabs}
\usepackage{mathtools,amsmath,fancyhdr,appendix,nccmath}
\usepackage[all]{xy} 

\usepackage{array}
\usepackage[left=2.8cm,top=2.8cm,right=2.8cm,bottom=2.8cm]{geometry}

\usepackage{hyperref}
\hypersetup{
    colorlinks=true,
    linkcolor=red,
    urlcolor=black,
    citecolor=green,
    hyperindex=true,
    bookmarks=true,
    pagebackref=true
}

\newtheorem{question}{Question}
\newtheorem{problem}{Problem}
\theoremstyle{plain}
\newtheorem{theorem}{Theorem}[section]
\newtheorem{lemma}[theorem]{Lemma}
\newtheorem{corollary}[theorem]{Corollary}
\newtheorem{proposition}[theorem]{Proposition}

\theoremstyle{definition}
\newtheorem{definition}[theorem]{Definition}
\newtheorem{example}[theorem]{Example}

\newtheorem{remark}[theorem]{Remark}
\theoremstyle{remark}

\def\fork{\mathrel{\raise0.2ex\hbox{\ooalign{\hidewidth$\vert$\hidewidth\cr\raise-0.9ex\hbox{$\smile$}}}}}

\newcommand{\p}{\mathfrak{p}}

\renewcommand{\a}{\mathfrak{a}}
\newcommand{\m}{\mathfrak{m}}

\newcommand{\Z}{\mathbb{Z}}
\newcommand{\ev}{_{\overline{0}}}
\newcommand{\od}{_{\overline{1}}}

\usepackage[mathscr]{euscript}

\usepackage{dsfont}

\newcommand{\ann}{\mathrm{Ann}}
\newcommand{\J}{\mathcal{J}}
\newcommand{\kdim}{\mathrm{Kdim}}
\newcommand{\ksdim}{\mathrm{Ksdim}}
\newcommand{\nil}{\mathrm{Nil}}
\newcommand{\sdim}{\mathrm{sdim}}

\usepackage{setspace}
\newcommand{\K}{\Bbbk} 

\begin{document} 
\setstretch{1}

\title{A note on unique Factorization in superrings} 

\author{Pedro  Rizzo }

\address{Instituto de Matemáticas, FCEyN\\Universidad de Antioquia}
\email{pedro.hernandez@udea.edu.co}

\author{Joel Torres del Valle }

\address{Instituto de Matemáticas, FCEyN\\Universidad de Antioquia}
 
\email{joel.torres@udea.edu.co}

\author{Alexander Torres-Gomez }

\address{Instituto de Matemáticas, FCEyN\\Universidad de Antioquia}
 
\email{galexander.torres@udea.edu.co}

\date{\today}

\maketitle

\begin{abstract} 
In the realm of supercommutative superrings, this article investigates the  unique factorization of elements. We build upon recent findings by Naser et. al. concerning similar results  in noncommutative symmetric  rings with zerodivisors, delving deeper into the ramifications. Strikingly, we demonstrate that any unique  factorization superdomain necessarily takes the form of a superfield,  further characterizing them as Artinian superrings with Krull superdimension $0\mid d$. Furthermore, we discover that a straightforward analogue   of the well-known Auslander-Buchsbaum Theorem does not hold true in the supercommutative setting.  
\\

\end{abstract} 


\section{Introduction}


The fundamental concept of factorization has long been a cornerstone of algebraic structures, particularly  within the realm of commutative rings. At its heart lies the notion of  unique factorization rings, where every non-zero, non-unit element can be expressed uniquely (up to units and multiple factors order) as the product  of irreducible elements. This concept has been extensively explored across diverse settings: commutative rings without zerodivisors (\cite[p.13]{eisenbud}), noncommutative rings without zerodivisors (\cite{cohn1963noncommutative}), commutative rings with zerodivisors (\cite{galovich1978unique}) and, most notably, noncommutative symmetric rings with zerodivisors (\cite{naser2021non}). Yet, within the domain of supercommutative superrings, this facet of factorization remains uncharted territory. Motivated by \cite{naser2021non}, this work embarks on a parallel journey for superrings (where the noncommutative rings are $\Z_2$-graded and supercommutative but not symmetric in general), replicating the results and unveiling their intriguing consequences. One striking departure from the familiar terrain of commutative algebra emerges: unlike the usual characterization of Dedekind domains as unique factorization domains if and only if they are principal ideal domains, this equivalents crumbles in the supercommutative context. Indeed, we prove that a unique factorization superdomain most necessarily be a superfield  (Theorem \ref{thm:3.3} i)). Furthermore, focusing on the local case, we reveal that any  Noetherian unique factorization superring must be an Artinian superring with Krull superdimension  $0\mid d$, where $d$ is a non-negative integer (Theorem \ref{thm:sartin}). Notably, we discover that  an analogue of the well-known Auslander-Buchsbaum Theorem does not hold true in the supercommutative setting (Subsection \ref{AUSLANDER}).  

This paper is structured as follows. Section \ref{Section:2} establishes the requisite background material for our subsequent analysis, introducing key definitions and results from the theory of superrings.  Section \ref{Section:3} embarks on an rigorous exploration of superrings endowed with unique factorization, forging a crucial connection between this property and the defining characterization of Artinian superrings. Additionally, we present several illustrative examples and introduce a new property of superrings termed {\it oddly-Noetherian} (see the previous comment to Theorem \ref{Theorem:3.7}). Finally, Section \ref{comments} delves into a nuanced discussion of the ramifications and broader implications gleaned from the preceding section's results. 

\section{Background material}\label{Section:2}

Throughout this paper, we work with the group of integers module 2, denoted by $\Z_2$. We restrict our attention to unitary and  non-trivial rings (i.e., $1\neq 0$). Furthermore, we assume that $2$ is a non-zerodivisor in every ring considered.

\begin{definition}\label{Definition:2.1}
    A $\Z_2$-graded unitary ring $R=R\ev\oplus R\od$ is called a \emph{superring}. 
\end{definition} 

Let $R$ be a superring. We define the \textit{parity} of $x\in R$ as

\[
|x|:=\begin{cases}
    0& \text{ if }\quad x\in R\ev,\\
    1& \text{ if }\quad x\in R\od.
\end{cases}
\]

The set $h(R):=R\ev\cup R\od$ is called the \textit{homogeneous elements of $R$} and each $x\in h(R)$ is said to be  \textit{homogeneous}.  A homogeneous nonzero element $x$ is called \emph{even} if $x\in R\ev$ and \textit{odd} if $x\in R\od$. 

Any superring $R$ in this paper is  \textit{supercommutative}, i.e., 

\begin{equation}\label{Supercommutativity}
    xy=(-1)^{|x||y|}yx,\quad \text{ for all }\quad x, y\in h(R).
\end{equation}

\noindent Note that from \eqref{Supercommutativity}, it immediately follows that $R\ev$ is a commutative ring and every element of $R\od$ squares to zero. 

\begin{definition}\label{Definition:2.2}
Let $R$ be a superring. 

\begin{enumerate}[itemsep=2pt]
\item[i)] A \textit{superideal} of $R$ is a $\Z_2$-graded ideal  $\a$ of $R$, that is, it admits a decomposition $$\a=(\a\cap R\ev)\oplus(\a\cap R\od).$$
\item[ii)] An ideal $\p$ of $R$ is \textit{prime} (resp. \textit{maximal}) if $R/\p$ is an integral domain (resp. a field).
\end{enumerate}
\end{definition}

\begin{remark}\label{Remark:2.3} Let $R$ be a superring. 

\begin{enumerate}[itemsep=2pt]
\item[i)] If $\a$ is a superideal of $R$, we denote $\a_i:=\a\cap R_i$, for each $i\in\Z_2$.
\item[ii)]  Any superideal $\a$ of $R$ is a two-sided ideal.
\item[iii)] If $\a$ is a superideal of $R$, it is easy to show that the quotient ring $R/\a=(R/\a)\ev\oplus(R/\a)\od$ becomes a superring with $\Z_2$-gradding given by $(R\ev/\a\ev)\oplus(R\od/\a\od)$.
\item[iv)] The definition of prime ideal is equivalent to requiring that if $x, y$ are (homogeneous) elements in $R$ whose product is in $\p$, one of them is in $\p$ (\cite[Lemma 4.1.2]{westrathesis}). In other words, any prime ideal is completely prime.  
\item[v)] Any prime ideal of $R$ is a superideal of the form $\p=\p\ev\oplus R\od$, where $\p\ev$ is a prime ideal of $R\ev$. Similarly, any maximal ideal of $R$ is a superideal of the form $\m=\m\ev\oplus R\od$, where $\m\ev$ is a maximal ideal of $R\ev$ (\cite[Lemma 4.1.9]{westrathesis}).
\end{enumerate}
\end{remark}

\begin{definition}\label{Definition:2.4} Let $R$ be a superring. 

\begin{itemize}[itemsep=2pt]
    \item[i)] The ideal $\J_R=R\cdot R\od=R\od^2\oplus R\od$ is called the \textit{canonical superideal of $R$}.
    \item[ii)] The \textit{superreduction of} $R$ is the commutative ring $\overline{R}=R/\J_R$. 
    \item[iii)] $R$ is called a \textit{superdomain} if $\overline{R}$ is a domain or, equivalently, $\J_R$ is a prime ideal. 
    \item[iv)] $R$ is called a \textit{superfield} if $\overline{R}$ is a field or, equivalently, $\J_R$ is a maximal ideal.
        \item[v)] The \textit{Jacobson radical} of $R$, denoted by $\mathfrak{r}_R$, is defined as the  intersection of all maximal ideals of $R$.
    \item[vi)]  $R^{\times}$ denotes the \textit{group of units} of $R$.
    \item[vii)] $R$ is said to be  \emph{local} if it has a unique maximal ideal, or, equivalently, $R-R^{\times}$ is an ideal.
    \item[viii)] $\nil(R)$ represents the ideal of nilpotent elements of $R$. 
\end{itemize}
\end{definition} 

\begin{remark}
If $R$ is a superdomain, then $\nil(R)=\J_R$.
\end{remark}


\section{Main findings}\label{Section:3}


This section presents the principal findings of our work, divided into two subsections (unique factorization superrings and Artinian superrings). We first introduce the notion of unique factorization superrings (UFSR). We then prove that if a superdomain $R$ is a UFSR, then every irreducible element of $R$ is nilpotent. As a direct consequence, $R$ is a local superring with maximal ideal $\mathfrak{r}_R$, which is  nilpotent when $R$ is also \textit{oddly-Noetherian}. In the second subsection, we prove that any Noetherian superdomain that is a UFSR is also an Artinian superring.  

\begin{definition}\label{Definition:2.5}
    Let $R$ be a superring and let $a, b, p\in R$.

\begin{itemize}[itemsep=2pt]
    \item[i)]  We say that $a$ \textit{divides} $b$ (or $a$ is a \textit{factor} of $b$), denoted $a\mid b$, if there exist $c, d\in R$ such that $b=cad$.
    \item[ii)]  We say that $a$ is \textit{regular} if it is not a zerodivisor of $R$. In other words, for every nonzero $x\in h(R)$, the product $ax\neq0$ (cf. \cite[Corollary 3.1.4]{westrathesis}).
    \item[iii)] We say that $a$ is \textit{associated} to $b$ if there exist $u, v\in R^{\times}$ such that $a=ubv$. 
    \item[iv)] We say that $a$ is \textit{normal} if $aR=Ra$
    \item[vi)] We call a normal element $p\in R$ a \textit{prime} element if, for any $a, b\in R$; $p\mid ab$ implies that either $p\mid a$ or $p\mid b$.
    \item[vii)]  A non-unit element is said to be \textit{irreducible} if it cannot be expressed as the product of two non-unit elements.
\end{itemize}
\end{definition}

\subsection{Unique factorization superrings} \textit{ }
\medskip

Henceforth, we restrict our attention to superrings $R$ with a non-trivial odd part, meaning, $R\od\neq0$. This assumption is motivated by the fact that if $R\od=0$, then $R$ reduces to a commutative ring with zerodivisors, a class  studied in  \cite{galovich1978unique}. For defining unique factorization in superrings, we can leverage the concept from noncommutative ring presented in \cite{naser2021non} to supercommutative context.

\begin{definition}\label{def:ufsr}
    A \textit{unique factorization superring} (UFSR) is a superring $R$ satisfying the following two conditions:

    \begin{itemize}[itemsep=2pt]
        \item[i)]  Every nonzero, non-unit element $x\in R$ admits a finite factorization $$x=f_1\cdots f_d$$ into  normal irreducible elements 
        $f_1, \ldots, f_d$ of $R$. These elements $f_i$ are called the \textit{factors} of $x$.
        
        \item[ii)]  If $x$ possesses two distinct factorizations into normal irreducible elements as in i),  say $$x=f_1\cdots f_d=g_1\cdots g_n,$$ then $d=n$ and there exists a permutation $\sigma\in\mathbb{S}_n$, such that $f_i$ is associated to $g_{\sigma(i)}$, for all $i=1, 2,\ldots, n$.
    \end{itemize}
\end{definition}

To introduce some illustrative examples, we first define polynomial superrings. Let $\K$ be a field. We then consider the following ring

\begin{eqnarray}\label{eqn:kalg}
    R & = & \K[X_1, \ldots, X_s \mid  \theta_1, \ldots, \theta_d] \nonumber \\ 
      &:= & \K\langle Z_1, \ldots, Z_s, Y_1, \ldots, Y_d\rangle/(Y_iY_j+Y_jY_i, Z_iZ_j-Z_jZ_i, Z_iY_j-Y_jZ_i).
\end{eqnarray}

This $\K$-superalgebra is called the \textit{polynomial superalgebra over} $\K$, with \textit{even indeterminants} $X_i$'s and \textit{odd indeterminants} $\theta_i$'s. The element $\theta_i$ corresponds to the image of $Y_i$ and $X_j$ corresponds to the image of $Z_j$ in the quotient (\ref{eqn:kalg}). \\

The super-reduced of $R=\K[X_1, \ldots, X_s \mid  \theta_1, \ldots, \theta_d]$ is $\overline{R}\cong\K[X_1, \ldots, X_s]$ and it is not hard to check that any element of $R$ can be written in the form 

\vspace{0.2cm}

\[
f=f_{i_0}(X_1, \ldots, X_s)+\sum_{J\,:\,\text{even}}f_{i_1\cdots i_J}(X_1, \ldots, X_s)\theta_{i_1}\cdots\theta_{i_J}+\sum_{J\,:\,\text{odd}}f_{i_1\cdots i_J}(X_1, \ldots, X_s)\theta_{i_1}\cdots\theta_{i_J},
\]

\vspace{0.2cm}

\noindent where $f_{i_0}, f_{i_1\cdots i_J}\in\K[X_1, \ldots, X_s]$ for all $J$. 

\begin{example}[The dual numbers]\label{Ex:3.3}
    Let $\mathbb{R}$ be the field of real numbers. We consider $R=\mathbb{R}[\varepsilon]$, where $\varepsilon$ is an odd variable. Elements of $R$ are then of the form $\alpha=\alpha_0+\alpha_1\varepsilon$, where $\alpha_0, \alpha_1\in\mathbb{R}$. Here, the even and odd elements are $$R\ev=\mathbb{R}\quad\text{and}
    \quad R\od=\varepsilon\mathbb{R}.$$ Furthermore, the set of units $$R^\times=\{\alpha_0+\alpha_1\varepsilon\in R\mid \alpha_0\neq0\}.$$ The inverse of an element $\alpha=\alpha_0+\alpha_1\varepsilon$ (with $\alpha_0\neq0)$ is given by $$\alpha^{-1}=\alpha_0^{-1}-[\alpha_1(\alpha_0^{-1})^2]\varepsilon.$$ The non-units of $R$ are those elements of the form $\alpha
        _1\varepsilon$, with $\alpha_1\in\mathbb{R}$. Since $\varepsilon^2=0$, any non-unit is irreducible. Therefore, $R$ is a UFSR. 
\end{example}

The following example generalizes the construction of dual numbers from Example \ref{Ex:3.3} to an arbitrary field and number of odd variables.

\begin{example}\label{Ex:3.4} 
    Let $\Bbbk$ be a field and consider $N$ odd variables $\theta_1, \ldots, \theta_N$.  We define a ring $R$ as follows

        \begin{eqnarray}\label{eqn:dualgen}
            R &=&\Bbbk[\varepsilon_1, \varepsilon_2, \ldots, \varepsilon_N] \nonumber\\
            &:=&\Bbbk[\theta_1, \theta_2, \ldots, \theta_N]/(\theta_i\theta_j\mid i, j=1, \ldots, N), 
        \end{eqnarray}

         \noindent where each $\varepsilon_j$ corresponds to the image of $\theta_j$ in the quotient (\ref{eqn:dualgen}). This ring consist of elements of the form $$\alpha=\alpha_0+\alpha_1\varepsilon_1+\cdots+\alpha_N\varepsilon_N,\quad \text{where }\quad \alpha_0, \alpha_1, \ldots, \alpha_N\in\Bbbk.$$ We have $$R\ev=\Bbbk\quad\text{and}\quad R\od=\varepsilon_1\Bbbk+\cdots+\varepsilon_N\Bbbk.$$ 
        An element $\alpha$ in $R$ is a unit if $\alpha_0\neq0$. The inverse of such a unit is given by $$\alpha^{-1}=\alpha_0^{-1}-[\alpha_1(\alpha_0^{-1})^2]\varepsilon_1-\cdots-[\alpha_N (\alpha_0^{-1})^2]\varepsilon_N.$$ The non-units of $R$ are those elements where  $\alpha_0=0$. Since $\varepsilon_i\varepsilon_j=0$ for all $i, j=1, 2, \ldots, N$,  any non-unit is irreducible. Therefore, $R$ is a UFSR. This property holds true even if we consider an infinite number of odd variables $\varepsilon_i$'s. 
\end{example} 

\begin{example}[A non-supercommutative UFSR]
     Let 

         \[
         R=\begin{pmatrix}
             \mathbb{Z}&\mathbb{Q}\\0&\mathbb{Z}
         \end{pmatrix}=\left\{\left.\begin{pmatrix}
             a_{11}&a_{12}\\0&a_{22
             }
         \end{pmatrix}\right\vert a_{12}, a_{22}\in\mathbb{Z}, a_{12}\in\mathbb{Q}\right\}.
         \]

             \noindent Thus, $R$ is a unique factorization ring \cite[Example 12]{naser2021non}. Furthermore, it is not hard to see that $R$ is a (non-supercommutative) superring with

         \[
         R\ev=\begin{pmatrix}
             \mathbb{Z}&0\\0&\mathbb{Z}
         \end{pmatrix}\quad\text{and}\quad R\od=\begin{pmatrix}
             0&\mathbb{Q}\\0&0
         \end{pmatrix}.
         \]

         \noindent Note that $R\od^2=0$.
\end{example}

\begin{example}\label{Example:3.6} 
    Consider odd variables $\theta_1, \theta_2$ and let $R=\Z_2[\theta_1, \theta_2]$. Then, 

    \begin{align*}
        R&=\{0, 1, \theta_1\theta_2, 1+\theta_1\theta_2, \theta_1, \theta_2, \theta_1+\theta_2, 1+\theta_1, 1+\theta_2, 1+\theta_1+\theta_2,\theta_1+\theta_1\theta_2, \theta_2+\theta_1\theta_2, \\ 
        &\hspace{4cm} \theta_1+\theta_2+\theta_1\theta_2,1+\theta_1+\theta_1\theta_2, 1+\theta_2+\theta_1\theta_2, 1+\theta_1+\theta_2+\theta_1\theta_2\}.
    \end{align*}

    \noindent Then, $R$ is clearly a UFSR.

    \begin{itemize}
        \item The even elements are $$R\ev=\{0, 1, \theta_1\theta_2, 1+\theta_1\theta_2\}.$$
        
        \item The non-units of $R\ev$ are $0$ and $\theta_1\theta_2$ which is irreducible in $R\ev$. Hence, $R\ev$ is a UFR in the sense of Galovich \cite{galovich1978unique}.
\item The set of units of $R$ is 

\begin{align*}
    R^\times&=\{1, 1+\theta_1\theta_2, 1+\theta_1, 1+\theta_2, 1+\theta_1+\theta_2, 1+\theta_1+\theta_1\theta_2,\\
    &\hspace{7cm} 1+\theta_2+\theta_1\theta_2, 1+\theta_1+\theta_2+\theta_1\theta_2\}.
\end{align*}
        
        \item The non-units of $R$, 
        
        $$R-R^\times=\{0, \theta_1\theta_2, \theta_1, \theta_2, \theta_1+\theta_2,\theta_1+\theta_1\theta_2, \theta_2+\theta_1\theta_2, \theta_1+\theta_2+\theta_1\theta_2\},$$ 
        
        \noindent except $\theta_1\theta_2$, are irreducible.  Furthermore,  $\theta_1\theta_2$ is the product of normal irreducible elements, non-uniquely determined. Namely, $\theta_1\theta_2=\theta_1(\theta_1+\theta_2)$. Therefore,  $R$ is not a UFSR. 
    \end{itemize}
\end{example}

\begin{remark}\label{rmk:ufsrelem}
Drawing upon Propositions 3.4, 3.5, 3.6 and Lemma 3.7 in \cite{naser2021non}, we establish the following properties of UFSR.   
 
    \begin{itemize}[itemsep=2pt]
    \item[i)] A normal element  $p\in R$ is prime if and only if the ideal $pR$ is prime.

    \item[ii)] Every normal  irreducible element is prime. 
     
    \item[iii)] Every UFSR contains at least one nonzero normal irreducible element that is also a zerodivisor.
     
    \item[iv)] Every normal irreducible element is a zerodivisor. 
\end{itemize}
\end{remark}

\begin{remark}
The uniqueness of factorization is crucial for the conclusions developed in this work. Consider the ring $R=\Z[\varepsilon]$, where $\varepsilon$ is an odd variable. The units are then $$R^{\times}=\{\pm 1+b\varepsilon\mid b\in \Z\}.$$ Thus, an element $a+b\varepsilon$ is irreducible if the absolute value $|a|$ is a prime number. It is not hard to see that $R$ satisfies Definition \ref{def:ufsr} i). However, it fails condition ii). For example, for any $p$ prime, we have $$p^2=(p-\varepsilon)(p+\varepsilon)=(p-p\varepsilon)(p+p\varepsilon),$$ but these factors are not associates. Note that, even though these factorizations exits, the normal irreducible elements of $R$ are not necessarily zerodivisors. 
\end{remark}

\begin{remark}\label{Remark:3.4:UFSR}
    A noncommutative ring $A$ is called \emph{symmetric} if $abc=0$ implies $acb=0$, for all $a, b, c\in A$. An ideal $\a$ of $A$ is called \textit{symmetric} if $abc\in\a$ implies $acb\in\a$, for all $a, b, c\in A$. This is equivalent to requiring that if $a_1a_2\cdots a_n\in \a$ and $\sigma\in\mathbb{S}_n$ is a permutation, then $a_{\sigma(1)}a_{\sigma(2)}\cdots a_{\sigma(n)}\in\a$ (\cite[Proposition 2.1]{lambek_1971}). Additionally, since any prime ideal of a superring is completely prime, it follows that any prime ideal in a superring is symmetric. 
\end{remark}

Building upon the work of  Galovich (\cite[Proposition 6]{galovich1978unique}), Naser et. al. (\cite[Theorem 3.11]{naser2021non}) proved that any noncommutative symmetric UFR is a local ring with Jacobson ideal being the nilpotent maximal ideal. This section aims to establish a similar result for superrings, but under the more natural condition of \textit{oddly-Noetherian} instead of symmetry.
\medskip

The following lemma will be useful later. 

\begin{lemma}\label{Lemma:3.10} Let $R$ be a unique factorization superring, and $a\in R$ be an even irreducible element. Let $x\in R$ be a non-zero homogeneous element such that $ax=0$. Then, there is a non-negative integer $n$ such that $x=a^ny$, where $y=f_1\cdots f_d$, where all the $f_i$ $(i=1, 2, \ldots, d)$ are not associates normal irreducible elements, and none of them is associated to $a$.
\end{lemma}

\begin{proof}
Since $x \in R$ is a non-zero, non-unit element, it can be factored as $x=g_1\cdots g_s$ into normal irreducible elements. Assume that some $g_i$ ($i=1, 2, \ldots, s$) is associated to $a$. Then, there exist $u_i, v_i\in R^\times$ such that $g_i=u_iav_i$. Using that $a$ is even,  we obtain 

    \begin{align*}
         x&=g_1\cdots g_{i-1}g_i g_{i+1}\cdots g_s\\
         &=g_1\cdots g_{i-1}u_i a v_i g_{i+1}\cdots g_s\\
         &=ag_1\cdots g_{i-1}u_iv_i g_{i+1}\cdots g_s.
    \end{align*}
 \noindent By repeatedly using this logic to each $g_j$  associated to $a$, we obtain an expression $$x=a^n\cdots g_{i+1}u_iv_ig_{i+1}\cdots u_jv_j\cdots,$$ where the remaining $g$'s are not associated to $a$ and the $u$'s and $v$'s are units of $R$. Defining $f_k$ as products of units and a remaining $g$, we find that $$x=a^nf_1\cdots f_d,$$ where the $f_i$'s are normal irreducible elements not associated to $a$, and not associated to each other.  This completes the proof of the lemma. 
\end{proof}

Examples \ref{Ex:3.3}, \ref{Ex:3.4} and \ref{Example:3.6} somehow suggest that every unique factorization superring with no ``purely even'' zerodivisors is a superfield. i) of the following theorem gives a more precise statement of this assertion.

\begin{theorem}\label{thm:3.3}
Let $R$ be a non-trivial unique factorization  superdomain.

\begin{enumerate} 
\item[\rm i)] $R$ is a local ring with maximal ideal $\J_R$. In particular, $R$ is a superfield.

\item[\rm ii)] Every normal and irreducible element of $R$ is nilpotent.
\end{enumerate} 
\end{theorem}
 
\begin{proof}  \textit{ }

\begin{itemize} 
    \item[i)] Assume condition ii) holds. Suppose, for a contradiction, that $x\in R-\J_R$ is a non-unit. Then, by unique factorization,  there exist finitely many irreducible elements $f_1, \ldots, f_n$ in $R$ uniquely determined such that $x=f_1\cdots f_n$. Since each element $f_i$ is normal and irreducible, it must be nilpotent by ii), for every $i=1,\ldots, n$. Now, because $R$ is a superdomain, $\nil(R)=\J_R$, so $f_i\in\J_R$ for all $i$, \textit{a fortiori}, $x\in\J_R$, contradicting our initial assumption about $x$. Therefore, all elements in $R-\J_R$ must be units, implying that $\J_R$ is the unique maximal ideal of $R$ and making  $R$ a superfield, as desired.

    \item[ii)] Let $a\in R$ be a normal irreducible element.  If $a$ is odd, then $a^2=0$ (making it a nilpotent). So, let us assume that $a$ is not homogeneous. We claim that  $a\ev$ is irreducible. Assume  $a\ev$ factors as 
    
    $$a\ev=f_1\cdots f_d.$$ 
    
    \noindent This implies,  $$a=f_1\cdots f_d+a\od\quad\text{and}\quad a\od a=a\od f_1\cdots f_d.$$ Furthermore, let the factorization of $a\od$ be $$a\od=g_1\cdots g_s.$$ We obtain  $$g_1\cdots g_s a=g_1\cdots g_s f_1\cdots f_d.$$ Unique factorization then dictates that $s+1=s+d$, leading to $d=1$ and $a\ev=f_1$. Therefore, if $a$ is irreducible, $a\ev$ is also irreducible. In this context, suppose that for some non-negative integer $n$, $a\ev^n=0$. Consequently, $a^{n+1}=0$. Thus, the proof of the theorem is complete upon demonstrating that any even irreducible element is nilpotent.   
        
    Let $a\in R$ denotes an even irreducible element. Then, $a$ is both  prime and zerodivisor (Remark \ref{rmk:ufsrelem} ii) and iv)). Therefore, there exists a nonzero homogeneous element $x\in R$ such that $ax=0$. By Lemma \ref{Lemma:3.10}, we can express $x$ as $x=a^ny$, with $n\geq0$ and $y=f_1\cdots f_s$. Here, the $f_i$ ($i=1, 2, \ldots, s$) are not associated normal irreducible elements, and none of them is associated to $a$. Note that $a^{n+1}y=0$. We claim that $$a^{2(n+1)}=0.$$ To arrive at a contradiction, let's assume $a^{2(n+1)}\neq0$. Then $$a^{2(n+1)}=a^{n+1}(a^{n+1}+y)$$ implies that $a^{n+1}$ divides $y$, a contradiction. We can therefore conclude that $$a^{2(n+1)}=0 \, ,$$ establishing that $a$ is nilpotent, as desired.  
\end{itemize}
\end{proof}

\begin{example}\label{Example:3.12} This example shows that the property of being a superfield is not sufficient to guarantee the existence of unique factorization. In other words, a superfield may not be a UFSR. Consider   $R=\Bbbk[\theta_1, \theta_2, \theta_3]$, where $\Bbbk$ is a field and $\theta_1,\theta_2, \theta_3$ are  odd variables. We then define the quotient ring

\[S=\Bbbk[\varepsilon_1, \varepsilon_2, \varepsilon_3]:=R/(\theta_1\theta_2-\theta_1\theta_3).\]
       Since $\varepsilon_1\varepsilon_2=\varepsilon_1\varepsilon_3$, we obtain $\varepsilon_1\varepsilon_2\varepsilon_3=0$. The even and odd parts of $S$ are given by
    
    \[S\ev=\Bbbk+ \Bbbk \varepsilon_1\varepsilon_2+\Bbbk \varepsilon_2\varepsilon_3\quad\text{and}\quad S\od=\Bbbk\varepsilon_1+\Bbbk \varepsilon_2+\Bbbk \varepsilon_3.
    \]

Additionally, the canonical ideal is $\J_S=\Bbbk\varepsilon_1+\Bbbk \varepsilon_2+\Bbbk \varepsilon_3 + \Bbbk \varepsilon_1\varepsilon_2+\Bbbk \varepsilon_2\varepsilon_3$. Therefore, $S$ is a superfield, because $S/\J_S \simeq \Bbbk$.
\medskip

An element $\alpha \in S$ 
    
    \[\alpha=\alpha_0+\alpha_1\varepsilon_1+\alpha_2\varepsilon_2+\alpha_3\varepsilon_3+\alpha_{12}\varepsilon_1\varepsilon_2+\alpha_{23}\varepsilon_2\varepsilon_3,\]
    
    \noindent where $\alpha_0, \alpha_1, \alpha_2, \alpha_3, \alpha_{12}, \alpha_{23} \in \Bbbk$, is invertible if and only if $\alpha_0\neq0$. In such a case, the inverse of $\alpha$, denoted by $\alpha^{-1}$, can be determined as
    \[
    \alpha^{-1}=\gamma_0+\gamma_1\varepsilon_1+\gamma_2\varepsilon_2+\gamma_3\varepsilon_3+\gamma_{12}\varepsilon_1\varepsilon_2+\gamma_{23}\varepsilon_2\varepsilon_3,\]
    
    \noindent where 
    $$
    \gamma_0:=\alpha_0^{-1};\,\,\gamma_1:= - \alpha_0^{-2} \alpha_1;\,\, \gamma_2:=-\alpha_0^{-2} \alpha_2;\,\, \gamma_3:=-  \alpha_0^{-2}\alpha_3;
    $$
    
    $$
    \gamma_{12}:=2 \alpha_0^{-3} \alpha_1 \alpha_2- \alpha_0^{-2} \alpha_{12};\,\, \gamma_{23}:=2 \alpha_0^{-3} \alpha_2 \alpha_3- \alpha_0^{-2} \alpha_{23}.
    $$
\vspace{0.1cm}  

\noindent Moreover, the non-units of $S$ are precisely the elements $\alpha$ for which $\alpha_0=0$. While $\varepsilon_1, \varepsilon_2, \varepsilon_3$ constitute irreducible elements, the equation $\varepsilon_1\varepsilon_2=\varepsilon_1\varepsilon_3$ exposes the failure of unique factorization within $S$. 
\end{example}
 
Recall that a superring $R$ is said to be \textit{Noetherian} if and only if the even part $R\ev$ is Noetherian and the odd part $R\od$ is a finitely generated $R\ev$-module (\cite[Lemma 1.4]{MASUOKA2020106245}). We then introduce the term \textit{oddly-Noetherian superring} to describe a superring where $R\od$ is a finitely generated $R\ev$-module, regardless of the Noetherian property of $R\ev$. 

\begin{theorem}\label{Theorem:3.7}
\it Let $R$ be a non-trivial oddly-Noetherian unique factorization superdomain. Then, $R$ is local with unique maximal ideal $\J_R$, which is a nilpotent ideal.  
\end{theorem} 

\begin{proof}
By Theorem \ref{thm:3.3}, we note that $R$ is local with maximal ideal $\J_R$. Since $R$ is oddly-Noetherian, we can choose a minimal generating set $\{z_1, \ldots, z_n\}$ of the $R\ev$-module $R\od$. We will prove that  $\J_R$ is nilpotent. Fix the integer $m=n+1$. Now, consider any $x_1, \cdots, x_m\in\J_R$. For each $i\in\{1, \ldots, m\}$, there exist elements $a_{i1}, \ldots, a_{in}\in R\ev$ such that $$x_i=a_{i1}z_1+\cdots+a_{in}z_n.$$ Then, we obtain, for $x_1\cdots x_m\in\J_R^m$,

\vspace{0.3cm}

\begin{align}
    x_1\cdots x_m&=\prod_{1\leq i\leq m}\sum_{1\leq j\leq n}a_{ij}z_j\nonumber\\
    &=\sum_{1\leq j_1, \ldots,\, j_m\leq n}a_{1j_1}\cdots a_{ mj_m}z_{j_1}\cdots z_{j_m}\\
    &=0.\nonumber
\end{align}

\vspace{0.3cm}
 
\noindent (The product vanishes because $m>n$.)  
\end{proof}

\begin{corollary}\label{Corollary:3.7}
    Let $R$ be a non-trivial Noetherian unique factorization superdomain. Then, $R$ is local with maximal ideal $\J_R$, which is a nilpotent ideal.  
\end{corollary}

\begin{proof}
    Since any Noetherian superring is oddly-Noetherian by definition, the result follows from Theorem \ref{Theorem:3.7}.
\end{proof}
 
\subsection{Artinian superrings}\textit{ }
\medskip

This subsection delves into the relationship between UFSRs and Artinian superrings, elucidating their connections.  Recall that we are assuming that $R\od\neq0$.  

\begin{definition}  
A superring  $R$ is \textit{Artinian} if it satisfies the descending chain condition on superideals.
\end{definition}

We denote the usual Krull dimension in the commutative context by $\kdim(-)$.  
\medskip

The subsequent  definition follows  Masuoka and Zubkov (\cite[\S 4]{MASUOKA2020106245}).

\begin{definition}
   Let $R$ be a Noetherian superring such that $\kdim(R\ev)<\infty$ and let $a_1, \ldots, a_s\in R\od$. For any $I\subseteq\{1, \ldots, s\}$, let $a^I:=a_{i_1}\cdots a_{i_{n}}$, where $I=\{i_1, \ldots, i_{n}\}$ and the product is taken in $R$. The annihilator of $a^I$ in $R\ev$,  $$\ann_{R\ev}(a^I):=\{r\in R\ev\mid ra^I=0\},$$ forms an ideal of $R\ev$.  

\vspace{0.2cm}

\begin{itemize}
    \item[i)] If  $\kdim(R\ev)=\kdim(R\ev/\ann_{R\ev}(a^I))$, we say that $a_{i_1}, \ldots, a_{i_n}$ form a \emph{system of odd parameters of length $n$}.

    \vspace{0.2cm}

    \item[ii)] The greatest integer $m$ such that there exists a set of odd parameters of length $m$ is called the \emph{odd Krull superdimension} of $R$ and is  denoted by $m:=\ksdim\od(R)$. 

    \vspace{0.2cm}

    \item[iii)] The \emph{even Krull superdimension} of $R$ is the usual Krull dimension of $R\ev$ and is denoted $\ksdim\ev(R)$.

    \vspace{0.2cm}

    \item[iv)] The \textit{Krull superdimension} of $R$ is the pair $\ksdim\ev(R)\mid\ksdim\od(R)$. 
\end{itemize}
\end{definition}

\newpage

\begin{proposition}\label{Pro:3.6.UFSR}
    Let $R$ be a superring. The following conditions are equivalent.

\vspace{0.2cm}

    \begin{itemize}
        \item[\rm i)] $R$ is Artinian.

\vspace{0.2cm}

        \item[\rm ii)] \begin{itemize}
            \item[\rm a)] $R\ev$ is an Artinian ring. 
            
\vspace{0.2cm}

            \item[\rm b)] The $R\ev$-module $R\od$ is Artinian. 
        \end{itemize} 
        
\vspace{0.2cm}

        \item[\rm iii)] \begin{itemize}
            \item[\rm a)] $R$ is Noetherian. 
            
\vspace{0.2cm}

            \item[\rm b)] $R$ has Krull superdimension $0\mid d$, where $d$ is a non-negative  integer.   
        \end{itemize}   
    \end{itemize}
\end{proposition}

\begin{proof}
   See \cite[\S 1]{ZubkovKolesnikov}.
\end{proof}
 
\begin{remark}\label{REMARK:3.11:UFSR}
    Let $R$ be an Artinian superring. Then, its even part $R\ev$ is an Artinian ring (Proposition \ref{Pro:3.6.UFSR} ii) a)) and any prime ideal $\p\ev$ of $R\ev$ is maximal (\cite[Proposition 8.1]{atiyah}). Consequently, any prime ideal of $R$ is maximal ($\p=\p\ev\oplus R\od$). In particular, if $(R,\m)$ is local, then $\m$ must be the unique prime ideal of $R$. Conversely, if the only prime ideal of a Noetherian local superring $R$ is the maximal one, then, Proposition \ref{Pro:3.6.UFSR} iii) guarantees that the superring $R$ is Artinian.
\end{remark}

The following proposition is the superization of \cite[Proposition 8.6]{atiyah}.

\begin{proposition}\label{Proposition;3.8}
    Let $(R, \m)$ be a local Noetherian superring. Then, exactly one of the following conditions holds.

\vspace{0.2cm}

    \begin{itemize}
        \item[\rm i)] $\m^n\neq\m^{n+1}$ for all $n$.

\vspace{0.2cm}

        \item[\rm ii)] $\m^n=0$ for some $n$, in which case $R$ is Artinian. 
    \end{itemize}
\end{proposition}

\begin{proof} The conditions i) and ii) are mutually exclusive. 
Assume that i) does not hold, meaning $\m^n=\m^{n}\m$, for all $n$. Then, by the ``super'' Nakayama's Lemma (\cite[Proposition 1.1]{bartocci2012geometry}), we get that $\m^n=0$. Now, consider any prime ideal $\p$ of $R$. Then, clearly $0=\m^n\subseteq\p$. Taking radicals and using the fact that every prime ideal is radical (\cite[Lemma 4.1.13]{westrathesis}), we find that $\m\subseteq\p$, and consequently $\m=\p$ (because $\m$ is maximal). In other words, $\m$ is the only prime ideal of $R$, and $R$ is Artinian (Remark \ref{REMARK:3.11:UFSR}).  
\end{proof}

\begin{theorem}\label{thm:sartin}
Every Noetherian unique factorization superdomain $R$ is local and Artinian. Furthermore, 

\[
\ksdim(R)=0\mid d, \quad \text{where}\quad d\geq1.
\]

\end{theorem}

\begin{proof}
    By Corollary \ref{Corollary:3.7}, $R$ is local and its maximal ideal is nilpotent. Therefore,  Proposition \ref{Proposition;3.8} ii) guarantees that $R$ is Artinian. Now, using  Proposition \ref{Pro:3.6.UFSR} iii), we conclude that $\ksdim\ev(R)=0$ and $\ksdim\od(R)=d\geq0$. However, since we assume that $R\od\neq0$, then $d\geq1$. 
\end{proof}

\section{ Final comments}\label{comments}

We now examine the broader implications of our findings. 

\subsection{Dedekind superrings}\label{DEDEKIND} 
\textit{ }
\medskip

In commutative algebra a well-known result establishes that a Dedekind domain is a unique factorization domain if and only it is a principal ideal domain. However, this type of result does not hold for superrings, a Dedekind superring cannot be a UFSR. Firstly, unlike UFSRs which are inherently local, Dedekind superrings as defined in \cite{JT1} are not necessarily local. Secondly, even when restricting to local Dedekind superrings, their even Krull superdimension differs from UFSRs. Specifically, local Dedekind superrings have an even Krull superdimension equal to one, whereas UFSRs always have even Krull superdimension equal to zero.  

\subsection{Auslander–Buchsbaum theorem}\label{AUSLANDER}
\textit{ }
\medskip

Let $R$ be a local Noetherian superring with maximal ideal $\m$ and residue field $\Bbbk=R/\m$. It can be proved that $\kdim(R\ev)$ is finite (\cite[Corollary 11.11]{atiyah}) and any set of odd parameters has length at most $\dim_{\Bbbk}((\m/\m^2)\od)$ (\cite[Lemma 5.1]{MASUOKA2020106245}). Consequently, we have $$\ksdim\ev(R)\leq\dim_\Bbbk((\m/\m^2)\ev)\quad\text{and}\quad\ksdim\od(R)\leq\dim_\Bbbk((\m/\m^2)\od).$$ Masuoka and Zubkov (\cite[\S 5]{MASUOKA2020106245}) define a regular superring $(R, \m)$ as one with $$
\ksdim(R)=\sdim_\Bbbk(\m/\m^2),$$ where $$\sdim_\Bbbk(\m/\m^2):=\dim_\Bbbk((\m/\m^2)\ev)\mid\dim_\Bbbk((\m/\m^2)\od).$$
\medskip

In commutative algebra, the Auslander–Buchsbaum Theorem  guarantees that every regular local ring is a unique factorization domain (\cite[Theorem 5]{auslander1959unique}). Given this known result, it is natural to inquire whether an analogous relationship exist in the superrings context.

\begin{question}\label{Q:1}\rm 
   Is every regular local superring a UFSR?
\end{question}

No, it is not. We can construct a counterexample by considering a a regular local superdomain with Krull superdimension $n\mid m$ such that $n > 0$. If $R$ were a UFSR, it would also be Artinian (Theorem \ref{thm:sartin}), with even Krull superdimension $n=0$, a contradiction.
\medskip

Now consider the superring $R=\mathbb{C}[\theta_1, \theta_2]$, where $\theta_1, \theta_2$ are odd variables. Observe that $\theta_1\theta_2=\theta_1(\theta_1+\theta_2)$.  It is not hard to see that $\theta_i$  ($i=1, 2$) is not associated to $\theta_1+\theta_2$, and hence unique factorization fails (even though  $R$ is local and regular with superdimension $0\mid 2$). Therefore, the Auslander–Buchsbaum theorem does not have an analogue in the context of superrings. 

\subsection{Weaker definitions of UFSRs}\label{S:4.3}
\textit{ }
\medskip

Our result might suggest that the definition of UFSRs employed in this work is overly restrictive even though it aligns well with the classical notion. This raises the question: could weaker conditions lead to richer properties and a broader class of superrings satisfying them? We explore two such alternative definitions:

\begin{itemize}
    \item \textbf{Homogeneous Element Factorization: } Define a superring as a \textit{homogeneous-UFSR} if it exhibits unique factorization for its homogeneous elements; that is, if every non-zero, non-unit homogeneous  element can be expressed uniquely as a product of homogeneous normal irreducible  elements.  
    
    \item \textbf{Even Element Factorization:} Define a superring $R$ as an \textit{even-UFSR} if its even part $R\ev$ is a unique factorization ring  with zerodivisors as defined by Galovich \cite{galovich1978unique} (see Example \ref{Example:3.6}). 
\end{itemize}

Following the same reasoning as in the proof of Theorem \ref{thm:3.3}, it can be shown that all irreducible elements are nilpotent and both homogeneous-USFR and even-UFSR are  necessarily superfields. Thus, the proposed weaker definitions ultimately lead to equivalent restrictions on superrings compared to our initial unique factorization definition.

\subsection{Classification of superfields}\label{}
\textit{ }
\medskip

Theorem \ref{thm:3.3} asserts that any unique factorization superring $R$ is inherently a superfield, implying $R$ is canonically isomorphic to its total superring of fractions $K(R)$ (defined as the localization of $R$ at the set of non-zerodivisors of $R\ev$). Furthermore, $$R\ev \simeq K(R)\ev= K(R\ev) .$$ 

For commutative rings, examples satisfying this condition include Artinian rings and absolutely flat rings (defined by the property that every module is flat). Notably, when $R\ev$ is a local absolutely flat ring, we can conclude that $R\ev$ is a field itself (\cite[Problem 28, p.35]{atiyah}). 

Now, absolutely flat rings coincide with commutative von Neumann regular rings, characterized by various properties including those based on ideals and weak dimension. Additionally, noncommutative graded versions of these rings have garnered recent interest due to their applications and characterization pursuits (see e.g., \cite{L}). Motivated by these recent developments, we pose the following problem:

\begin{problem}
Classify all superfields R for which the canonical superideal $\mathcal{J}_R$ is nilpotent.
\end{problem}

Solving this problem would have significant implications for the understanding UFSRs. Notably, Theorem \ref{Theorem:3.7} shows that oddly-Noetherian  UFSRs provide concrete examples of such superfields.

\section*{Acknowledgements}

\subsection*{Funding}
The first and second authors were partially supported by CODI (Universidad de Antioquia, UdeA) through project numbers 2020-33713 and 2022-52654.



\begin{thebibliography}{30}

\bibitem{altman2013term}
A.~Altman and S.~Kleiman.
\newblock {\em A term of commutative algebra}.
\newblock Worldwide Center of Mathematics, 2013.

\bibitem{atiyah}
M.~F. Atiyah and I.~G. MacDonald.
\newblock {\em Introduction to commutative algebra.}
\newblock Addison-Wesley-Longman, 1969.

\bibitem{auslander1959unique}
M. Auslander   and D. Buchsbaum
\newblock {\em Unique factorization in regular local rings.}
\newblock Proceedings of the National Academy of Sciences, 45(5): 733-734, 1959.

\bibitem{bartocci2012geometry}
C.~Bartocci, U.~Bruzzo, and D.~Hern{\'a}ndez.
\newblock {\em The geometry of supermanifolds}, volume~71.
\newblock Springer Science \& Business Media, 2012.

\bibitem{cohn1963noncommutative}
P.~M. Cohn.
\newblock {\em Noncommutative unique factorization domains}.
\newblock Transactions of the AMS, 109 (2):313-331, 1963.

\bibitem{eisenbud}
D.~Eisenbud.
\newblock {\em Commutative Algebra: With a View Toward Algebraic Geometry.}
\newblock Springer New York, 1995.

\bibitem{naser2021non}
A.~Naser, M.~H. Fahmy and A.~M. Hassanein, 
\newblock \textit{Non-commutative unique factorization rings with zero-divisors.}
\newblock {Journal of Algebra and Its Applications}, 20(02):2150022, 2021.

\bibitem{galovich1978unique} 
S. Galovich.
\newblock \textit{Unique factorization rings with zero divisors.}
\newblock {Mathematics magazine}, 51(5):276-283, 1978. 

\bibitem{lambek_1971}
J.~Lambek.
\newblock \textit{On the Representation of Modules by Sheaves of Factor Modules.}
\newblock {Canadian Mathematical Bulletin}, 14(3):359–368, 1971.

\bibitem{L}
D. Lannstrom. 
\newblock \textit{A characterization of graded von Neumann regular rings with applications to Leavitt path algebras.} \newblock {Journal of Algebra}, 567: 91-113, 2021.


\bibitem{MASUOKA2020106245}
A.~Masuoka and A.N. Zubkov.
\newblock \textit{On the notion of Krull super-dimension.}
\newblock {Journal of Pure and Applied Algebra}, 224(5):106245, 2020.

\bibitem{JT1}
P.~Rizzo, J.~Torres, and A.~Torres-Gomez.
\newblock \textit{Dedekind Superrings and Related concepts.}
\newblock { \href{arXiv:2310.03822[math.RA]}{arXiv:2310.03822 [math.RA]}}, 2023.

\bibitem{westrathesis}
D.~B. Westra.
\newblock {\em Superrings and Supergroups}.
\newblock PhD Thesis. Universität Wien, Wien, 2009.

\bibitem{ZubkovKolesnikov}
A.~N. Zubkov and P.~S. Kolesnikov.
\newblock \textit{On dimension theory of supermodules, super-rings, and superschemes.}
\newblock {Communications in Algebra}, 50(12):5387--5409, 2022.

\end{thebibliography}

\end{document}